\documentclass[12pt]{amsart}
\usepackage{amsmath, color} 
\usepackage{amssymb} 
\usepackage{amsthm}
\usepackage{hyperref}
\usepackage{amscd}
\usepackage{cleveref}
\usepackage{tikz}
\usepackage[all]{xy}
\usepackage[T1]{fontenc}
\usepackage[latin9]{inputenc}
\usepackage{amstext}

\global\long\def\comult{\bar{e}}%

\global\long\def\dl{\operatorname{del}}%

\global\long\def\lk{\operatorname{link}}%

\numberwithin{equation}{section}
\newtheorem{theorem}{Theorem}[section]
\newtheorem{lemma}[theorem]{Lemma}

\newtheorem{corollary}[theorem]{Corollary}
\newtheorem{conjecture}[theorem]{Conjecture}

\theoremstyle{definition}
\newtheorem{definition}[theorem]{Definition} 
\newtheorem{remark}[theorem]{Remark}
\newtheorem{example}[theorem]{Example}

\def\f0{\mathbf{0}}

\newcommand\del{\operatorname{\text{del}}}

\DeclareMathOperator{\cod}{{codim}}


\begin{document}

\title
{Simplicial complexes with many facets are 
vertex decomposable}
\date{\today}

\author[A. Dochtermann]{Anton Dochtermann}
\address[A. Dochtermann]{Department
of Mathematics,
Texas State University, San Marcos, TX 78666, USA}
\email{dochtermann@txstate.edu}
\urladdr{\url{https://dochtermann.wp.txstate.edu/}}

\author[R. Nair]{Ritika Nair}
\address[R. Nair]{Department of Mathematics, Oklahoma State University, Stillwater, OK 74078, USA}
\email{ritika.nair@okstate.edu}
\urladdr{\url{https://sites.google.com/view/ritika-nair}}

\author[J. Schweig]{Jay Schweig}
\address[J. Schweig]{
Department of Mathematics, Oklahoma State University, Stillwater, OK 74078, USA}
\email{jay.schweig@okstate.edu}
\urladdr{\url{https://math.okstate.edu/people/jayjs/}}

\author[A. Van Tuyl]{Adam Van Tuyl}
\address[A. Van Tuyl]
{Department of Mathematics and Statistics\\
McMaster University, Hamilton, ON, L8S 4L8, Canada}
\email{vantuyl@math.mcmaster.ca}
\urladdr{\url{https://ms.mcmaster.ca/~vantuyl/}}

\author[R. Woodroofe]{Russ Woodroofe}
\address[R. Woodroofe]
{FAMNIT, University of Primorska, 6000 Koper, Slovenia}
\email{russ.woodroofe@famnit.upr.si}
\urladdr{\url{https://osebje.famnit.upr.si/~russ.woodroofe/}}

\begin{abstract}
Suppose $\Delta$ is a pure simplicial complex on $n$ vertices having dimension $d$ and let $c = n-d-1$ be its codimension in the simplex.
Terai and Yoshida proved that
if the number of facets of 
$\Delta$ is at least $\binom{n}{c}-2c+1$,
then $\Delta$ is Cohen-Macaulay.  We
improve this result by showing that
these hypotheses imply 
the stronger condition that $\Delta$ is vertex decomposable.  We give examples to show that
this bound is optimal, and that the conclusion cannot be strengthened to the class of matroids or shifted complexes. We explore an application to Simon's Conjecture and discuss connections to other results from the literature.
\end{abstract}
\maketitle

\section{Introduction}

Suppose $\Delta$ is a simplicial complex on
vertex set $[n] = \{1,\ldots,n\}$.  Recall that $\Delta$ is {\it pure} if
all the facets of $\Delta$ (the maximal elements of $\Delta$ under 
inclusion) have the same cardinality.  As
first shown by Terai and Yoshida \cite[Theorem 3.1]{Terai/Yoshida:2006}, if the number
of facets of $\Delta$ is ``large enough,''\! then $\Delta$ is Cohen-Macaulay.  In other words, the simplicial
complex $\Delta$ enjoys additional topological 
properties that can be detected by simply
counting the number of facets.

In this short note we strengthen Terai and Yoshida's results.
As in their paper, we express
our theorem in terms of the codimension of
$\Delta$ relative to the simplex on the same vertex set.  Recall that the {\it dimension} of $\Delta$
is given by $\dim\Delta = \max\{\dim F ~|~ F \in \Delta\}$, where
$\dim F = |F|-1$.   If $\Delta$ has $n$ vertices, then the \emph{codimension} of $\Delta$, denoted $\cod{\Delta}$, is given by $c = n - \dim\Delta - 1$ (the number of vertices in the complement of a facet). 
Our main result can now be stated as follows.

\begin{theorem}\label{maintheorem}
Let $\Delta$ be a pure simplicial complex on vertex set $[n]$ of codimension $c$. If $\Delta$ has 
at least $\binom{n}{c}-2c +1$ facets, then $\Delta$
is vertex decomposable.
\end{theorem}

The class of vertex decomposable simplicial complexes was introduced by Provan and Billera in
\cite{Provan/Billera:1980}, and have a recursive definition (see
Section 2 for the formal statement).  
Terai and Yoshida \cite[Theorem 3.1]{Terai/Yoshida:2006} showed that, under the same hypotheses as Theorem
\ref{maintheorem}, the simplicial 
complex $\Delta$ is
Cohen-Macaulay.  Their result was stated in terms of the multiplicity of the Stanley-Reisner ring, which agrees with the number of facets of $\Delta$ when $\Delta$ is pure. Since being vertex decomposable
implies being Cohen-Macaulay, our main theorem
strengthens their result.  
To prove  \cite[Theorem 3.1]{Terai/Yoshida:2006}, Terai and Yoshida used an algebraic approach, employing the Stanley-Reisner correspondence between simplicial complexes and square-free monomials.
  Our proof, in contrast, requires only simple tools from topological combinatorics.

The bound in Theorem~\ref{maintheorem} is tight.  In Example \ref{example tight} we describe a family of simplicial complexes on $n$ vertices of codimension $c$ with $\binom{n}{c} - 2c$ facets, each of which fails to be vertex decomposable. In Remark \ref{remark matroid} we also show that Theorem \ref{maintheorem} cannot be strengthened to conclude that $\Delta$ is a matroid complex or shifted complex (both of which are classes known to be vertex decomposable), even if we require $\binom{n}{c}-2$ facets.

It is not hard to see that a simplicial complex $\Delta$ on vertex set $[n]$ with dimension $d$ and codimension $c$ can have more than $\binom{n}{c}-2c$ faces of dimension $d$, but fail to be pure.  For instance, consider the $1$-dimensional complex consisting of a triangle and an isolated vertex. However, we observe in Lemma \ref{morefacets} that if we require at least $\binom{n}{c}-c$ facets of this top dimension, then purity is guaranteed. This leads to Corollary~\ref{corollary even more} for arbitrary simplicial complexes, which strengthens \cite[Theorem 2.1]{Terai/Yoshida:2006}.

We also interpret our results in terms of square-free monomial ideals (the language used by Terai and Yoshida).  In Corollary \ref{corollary linear} we describe a sufficient condition for a square-free monomial ideal to have a linear resolution, which may be useful for researchers working in combinatorial commutative algebra.

Our main result also has connections to other work involving decompositions of simplicial complexes. Simon's Conjecture is a well-known open question regarding the structure of {\it shellable} complexes (see Conjecture \ref{conjecture Simon}). Combining our results with those from \cite{Coleman-Dochtermann-Geist-Oh}, we see that any counterexample to Simon's Conjecture cannot have ``too many'' facets.  On the other hand, a result of Laso\'{n} from \cite{Lason:2013} provides a sufficient condition for a complex $\Delta$ to be vertex decomposable in terms of extremal properties of the $f$-vector of $\Delta$. In Example~\ref{example fvector} we discuss how Theorem \ref{maintheorem} relates to the results of Laso\'{n}.

Our paper is organized as follows.  In Section 2 we provide the necessary definitions and background, and in particular recall relevant properties of vertex 
decomposable simplicial complexes.  In Section
\ref{section proof} we prove Theorem \ref{maintheorem}.  In Section \ref{section conclusion}, we address the optimality of Theorem \ref{maintheorem},  interpret our result in the language of square-free monomial ideals, and discuss the connection to Simon's Conjecture and  Laso\'{n}'s result.\\

%%%%%%%%%%%%%%%%%%%%%%%%%%%%%%%%%%%%%%%%%%%%%%%%%%%%%

\section{Background}

In this section, we recall additional
relevant definitions and results needed to prove our main theorem.
We refer the reader to \cite{Herzog/Hibi:2011} or \cite{Villarreal:2015} for additional background on combinatorial commutative algebra, and to \cite{Bjorner:1995} for topological combinatorics.
We continue to use the notation from the introduction.

Let $\Delta$ be a simplicial complex
on vertex set $V$, which we typically take as $V = [n] := \{1,2, \dots, n\}$.  An element $F \in \Delta$ is called a {\it face}, and a face not properly contained in any other face is called a {\it facet}. For each $i \in [n]$ we always assume that $\{i\}$ is a face of $\Delta$, which we call a vertex and denote simply as $i$. 
We will use the notation
$\Delta = \langle F_1,\ldots,F_s\rangle$ 
to denote the complete list of distinct facets. After fixing
a vertex $x \in [n]$, we can define two other simplicial
complexes from $\Delta$.  Namely, the {\it link} of $\{x\}$ is the simplicial complex
$${\rm link}_\Delta(\{x\}) = \{G \in \Delta ~|~
x \notin G,~~ G \cup \{x\} \in \Delta\},$$
and the {\it deletion} of $\{x\}$ is the 
complex
$${\rm del}_\Delta(\{x\}) = \{G \in \Delta ~|~ x \notin G\}.$$
We will abuse notation and 
write ${\rm link}_\Delta(x)$ (respectively,
${\rm del}_\Delta(x)$) instead of ${\rm link}_\Delta(\{x\})$ (respectively, ${\rm del}_\Delta(\{x\})$). A vertex $x$ is a 
{\it shedding vertex} if for any facet
$F \in \Delta$ with $x \in F$, there exists
a vertex $y \not\in F$ such that 
$(F \setminus \{x\})\cup \{y\}$
is also a face of $\Delta$. Equivalently, $x$ is a shedding vertex if every facet of ${\rm del}_\Delta(x)$ is a facet of $\Delta$.

Vertex decomposable simplicial 
complexes are then defined 
recursively as follows.

\begin{definition}\label{defn:vd}
A simplicial complex $\Delta$ is 
{\it vertex decomposable} if either 
\begin{enumerate}
    \item $\Delta = \{\emptyset\}$, or
     $\Delta$ is  a simplex (i.e.,
    $\Delta = \langle F \rangle$ for 
    some $F \subseteq [n]$), or 
    \item there exists a shedding vertex
    $x \in [n]$ such that $\lk_\Delta(x)$
    and $\dl_\Delta(x)$ are vertex
    decomposable.
\end{enumerate}
\end{definition}

\begin{remark}
We take this definition from Bj\"orner and Wachs 
\cite{Bjorner/Wachs:1997}.
All the complexes we consider will be pure, and in this case our definition recovers the notion of vertex decomposability introduced by Provan and Billera in \cite{Provan/Billera:1980}.
\end{remark}

We next collect some results regarding vertex decomposable complexes that will be useful for our study. Recall that a simplicial complex $\Delta$ is a {\it cone} if there exists a vertex $v \in V$ such that $v \in F$ for every facet $F \in \Delta$. This implies that there exists a simplicial complex $\Gamma$ on vertex set $V \backslash \{v\}$ such that every face of $\Delta$ is either a face of $\Gamma$, or else of the form $F \cup \{v\}$ for some face $F \in \Gamma$. In this case we write $\Delta = \Gamma \ast v$.
In the statement
below, the {\it $k$-skeleton} of a pure
simplicial complex $\Delta$ is (for $k \leq \dim \Delta$) the 
simplicial complex whose facets
are the faces of $\Delta$ having 
dimension $k$.

\begin{lemma}\label{vdprop}

The following properties regarding vertex decomposability hold.
\begin{enumerate}
\item 
\cite[Proposition 3.1.1]{Provan/Billera:1980} All $0$-dimensional complexes are vertex decomposable. 
\item 
\cite[Proposition 3.1.2]{Provan/Billera:1980} A pure $1$-dimensional complex $\Delta$ is vertex decomposable if and only if $\Delta$ is connected.
\item
\cite[Lemma 3.10]{Woodroofe:2011b}
Any $k$-skeleton of a vertex decomposable complex $\Delta$ is vertex decomposable.
\item 
\cite[Proposition 2.4]{Provan/Billera:1980}
A cone $\Delta = \Gamma \ast v$ is vertex decomposable if and only if $\Gamma$ is vertex decomposable.
\end{enumerate}
\end{lemma}

We also record the following observation.

\begin{lemma}\label{lem:removeone}
Suppose $\Delta$ is a pure simplicial complex on $[n]$ with $n \geq 3$ obtained by removing a single facet from the boundary of a simplex. Then $\Delta$ is vertex decomposable.
\end{lemma}
\begin{proof}
Suppose $\Delta$ has vertex set $[n]$, and without loss of generality suppose the facet removed to obtain $\Delta$ is $\{1,2, \dots, n-1\}$. Then the facets of $\Delta$ consist of all subsets of $[n]$ of the form $\{n\} \cup F$, where $F$ is any $(n-2)$-subset of $[n-1]$. Hence $\Delta = \partial \Delta_{n-1} \ast n$. Here $\partial \Delta_{n-1}$ denotes the boundary of the simplex on $[n-1]$, which can be recovered as the $(n-3)$-skeleton of $\Delta_{n-1}$. From Lemma \ref{vdprop} we conclude that $\Delta$ is vertex decomposable.
\end{proof}

The following lemma is immediate from the recursive definition of vertex decomposability, and we will use it without comment.
\begin{lemma} \label{lem_vd}
If $\mathcal{F}$ is a family of simplicial complexes, so that every $\Delta \in \mathcal{F}$ 
is either a simplex or has a shedding vertex v with the property that both 
$\lk_\Delta(v)$ and $\dl_\Delta(v)$ are in $\mathcal{F}$, then every complex in $\mathcal{F}$ is vertex decomposable.
\end{lemma}

%%%%%%%%%%%%%%%%%%%%%%%%%%%%%%%%%%%%%%%%%%%%%%%%%%%%%%%%

\section{Proof of Main Result} \label{section proof}

In this section we provide the proof of Theorem \ref{maintheorem}. In what follows, we let $\Delta$
be a pure simplicial complex on $n$ vertices
with codimension
$c=n-\dim\Delta-1$. If $F$ is a set of $\dim \Delta + 1$ vertices that is not a face of $\Delta$, we say $F$ is an {\it antifacet}. We write $e(\Delta)$ for the number of facets of
$\Delta$ and $\comult(\Delta)$ for the number of antifacets of $\Delta$. Note that $\comult(\Delta) = \binom{n}{c}-e(\Delta)$.

We start with a technical lemma.
\begin{lemma}
\label{lem:DoubleCounting} If $\dim \Delta \geq 1$, $c\geq 2$ and $\comult(\Delta)\leq2c-1$,
then $\Delta$ has at least two vertices that each avoids at most $2c-3$ antifacets.
\end{lemma}

\begin{proof}
We proceed by double counting on the set
\[
\mathcal{P}=\left\{ (v,F):v\text{ a vertex of }\Delta,F\text{ an antifacet of \ensuremath{\Delta,v\notin F}}\right\} .
\]
Counting by antifacets, we see that \begin{equation}
    \left|\mathcal{P}\right|\leq c\cdot(2c-1) = 2c^2-c.\label{eq:countbyantfac}
\end{equation}  

Suppose for a contradiction that all but at most one vertex avoids more than $2c-3$ antifacets. Then
\begin{equation}
\left|\mathcal{P}\right|\geq(n-1)\cdot(2c-2)\geq(c+1)\cdot(2c-2)=2c^{2}-2.\label{eq:doublecountbyv}
\end{equation}
As $2c^{2}-c\geq2c^{2}-2$ holds only if $c\leq2$, we have reduced
to the case $c=2$. But in this case, the inequalities (\ref{eq:countbyantfac}) and 
(\ref{eq:doublecountbyv}) give us that $2\cdot3\geq(n-1)\cdot2$,
so we are left with the case $n=4$. 

Thus, the proof reduces to the case where $\Delta$ is a one-dimensional complex on $4$ vertices. In this case, $\comult(\Delta)\leq 3$ and $c = 2$, and so we need to show there must be at least two vertices that each avoids at most $2c-3 = 1$ antifacets. Viewing the antifacets as edges in a graph on four vertices, the antifacets give
a graph on four vertices
where every vertex has degree
at most two (since $\Delta$ cannot have an
isolated vertex).
We finish the proof with the following easily-verified observation: For any graph on four vertices with three or fewer edges,
where the degree of each vertex is
$\leq 2$, there are at least two vertices that are each disjoint from at most one edge. \end{proof}

\begin{remark} \label{rem:pure}
Note that if $\Delta$ has the property that all vertices are shedding, then for any $v$ we have that $\del_\Delta(v)$ is pure of the same dimension as $\Delta$. Thus, in this situation Lemma \ref{lem:DoubleCounting} implies that $\Delta$ has at least two vertices satisfying $\overline{e}(\del_\Delta (v)) \leq 2c-3$.
\end{remark}

We make some other easy observations. Recall that a \emph{minimal nonface} of $\Delta$
is a nonface of $\Delta$ that is minimal under inclusion with respect to this property.  
Note that if $\{v,w\}$ is a minimal nonface of $\Delta$ of cardinality two, then $w$ is \emph{not} a vertex of $\lk_\Delta(v)$.
The following lemma says that minimal nonfaces of cardinality two are rare when $\dim \Delta \neq 1$ and the number of antifacets is small enough (equivalently, if the number of facets is large enough relative to $n$).

\begin{lemma}
\label{lem:MissingFaces}If $\comult(\Delta)\leq2c$, then no
nonface of $\Delta$ has fewer than $\dim \Delta$ vertices, and at most
one (minimal) nonface has $\dim\Delta$ vertices.
\end{lemma}

\begin{proof}
A nonface of size $\dim\Delta -1$ is contained in $\binom{c+2}{2}$
antifacets. But $(c+2)(c+1)/2$ is greater than $2c$ for all $c$, so no such nonface can exist.

Each nonface with $\dim\Delta$ vertices is contained in $c+1$
antifacets. If there existed a pair of such nonfaces, they would be contained in at most
one common antifacet. But $2(c+1)-1$ is greater than $2c$, contradicting our assumptions. 
\end{proof}

\begin{corollary}\label{cor.antifacetlinkbound}
Suppose $v$ is a vertex of $\Delta$. If $\comult(\Delta)\leq2c-1$ and
$b$ is the codimension of $\lk_{\Delta}(v)$, then $\comult(\lk_{\Delta}(v))\leq2b-1$.
\end{corollary}

\begin{proof}
Note that the statement is trivial if $\dim\Delta \leq 1$ since in this case $\comult(\lk_{\Delta}(v)) = 0$.
Now since $\Delta$ is pure, we have that $\dim\lk_{\Delta}(v)=\dim\Delta-1$. 
If $\dim\Delta > 2$, then by Lemma \ref{lem:MissingFaces} we have that $\Delta$ has no nonface of size two, and hence $\lk_{\Delta}(v)$ has $n-1$ vertices by the comment
after Remark \ref{rem:pure}. From this we get that $b = (n-1) - (\dim\Delta-1) - 1 =c$. The result then follows
by noting that if $F$ is an antifacet of $\lk_{\Delta}(v)$, then $F\cup\{v\}$
is an antifacet of $\Delta$. 

If $\dim\Delta = 2$, then there are two cases.  If $\Delta$
has no nonface of size two, then the argument follows as above.
Otherwise, $\Delta$ has a nonface of size two, and 
then again by Lemma\ \ref{lem:MissingFaces} we have that $b=c-1$ (and hence also $c\geq1$), and there is a unique minimal
nonface $\{v,w\}$ having two vertices. 
Now, as before there are
$\comult(\lk_{\Delta}(v))$ antifacets of $\Delta$ containing $v$
but not $w$, and an additional $n-2$ antifacets of $\Delta$ containing
$\{v,w\}$. 
We obtain that $\comult(\lk_{\Delta}(v))\leq\comult(\Delta)-(n-2) \leq (2c-1) - (c+1) = 2b - c$,
which yields the desired result. 
\end{proof}

We restate the definition of a shedding vertex in the language of
antifacets.
\begin{lemma}
\label{lem:AntifacetShedding}A vertex $w$ of $\Delta$ is a shedding
vertex unless $w$ is contained in some facet $F$ with the property that, for each
vertex $u\notin F$, the set $(F\setminus\{w\})\cup\{u\}$ is an antifacet.
\end{lemma}

With this we have the following observation.
\begin{lemma}
\label{lem:NonsheddingGivesShedding}If $\comult(\Delta)\leq2c-1$,
then:
\begin{enumerate}
\item Any vertex that is contained in at least $c$ antifacets is a shedding
vertex.
\item If $w$ is not a shedding vertex, and $F$ is a facet as in Lemma~\ref{lem:AntifacetShedding},
then any vertex $v\in F\setminus\{w\}$ is contained in $c$ antifacets,
and in particular is a shedding vertex.
\end{enumerate}
\end{lemma}

\begin{proof}
If $w$ is not a shedding vertex, then $w$ is contained in some facet
$F$ so that $(F\setminus\{w\})\cup\{u\}$ is an antifacet for each
$u\notin F$. In particular, there are $c$ antifacets that do not
contain $w$, hence at most $c-1$ that do contain $w$. 

Moreover, since the set $F\setminus\{w\}$ is contained in $c$ antifacets,
the same holds for every $v\in F\setminus\{w\}$.
\end{proof}

We are now ready to prove our main result.

\begin{proof}[Proof of Theorem \ref{maintheorem}]
Recall that $\Delta$ is assumed to be a pure  simplicial complex on vertex set $[n]$ with codimension $c$, with at least $\binom{n}{c}-2c +1$ facets, so that $\comult(\Delta)\leq2c-1$. 
The statement can be checked
directly for all simplicial complexes on $n \leq 3$
vertices, so we suppose that $n > 3$. It is convenient to handle the cases when $c=0,1,n-2,n-1$ separately.
If $c=0$ or $1$, then $\Delta$ is a simplex, or a simplex boundary
with at most one face removed, and the result follows by Lemma \ref{lem:removeone}.
On the other hand, if $c=n-1$, then $\dim\Delta=0$, and the result
is immediate.

Moving forward, we can assume that
$c \leq n-2$ (or equivalently, that $\dim \Delta \geq 1$) and that $c \geq 2$.
 To prove the theorem we wish to find a shedding vertex $v$ so that $\lk_{\Delta}(v)$ and $\dl_{\Delta}(v)$ also satisfy the hypotheses of the theorem. Note that
we are using Lemma \ref{lem_vd}.
First observe that for any vertex $v$,  the complex ${\rm link}_\Delta(v)$ is pure of dimension $\dim \Delta - 1$, and also satisfies the hypotheses in Corollary \ref{cor.antifacetlinkbound}.

We next turn to the deletion. For this note that if $v$ is a shedding vertex, then by definition  $\dim\dl_{\Delta}(v)=\dim\Delta$. Since the number of vertices of $\dl_{\Delta}(v)$ is $n-1$, we have that the codimension of $\dl_{\Delta}(v)$ is $c-1$. Hence to complete the proof we need to find a shedding vertex $v$ satisfying  $\comult(\dl_{\Delta}(v)) \leq 2c-3$.

If every (or even all but one) vertex of $\Delta$ is a shedding
vertex, then we are done by Lemma~\ref{lem:DoubleCounting}. Otherwise,
Lemma~\ref{lem:NonsheddingGivesShedding} yields a shedding vertex
$v$ that is contained in at least $c$ antifacets, so that $\comult(\dl_{\Delta}(v))\leq2c-1-c$.
In either case, we have produced the desired shedding vertex and the result follows.
\end{proof}

\section{Consequences and
Concluding Remarks} \label{section conclusion}

We end with some additional comments 
and examples related to 
our main result.  We first give 
an example to show that the bound in Theorem
\ref{maintheorem} is optimal.

\begin{example}\label{example tight}

For any $2 \leq c \leq n-2$, we construct a pure simplicial complex $\Delta$ on $n$ vertices that has codimension $c$ 
and $\binom{n}{c}-2c$ facets, and that is 
not Cohen--Macaulay. Since any vertex decomposable complex is also Cohen--Macaulay, this
shows that the bound in Theorem \ref{maintheorem} is optimal.

We start with the complete $(n - c - 1)$-skeleton of the simplex on $n$ vertices, and obtain 
$\Delta$
by removing the facets containing the following vertices: all of 
$\{1, \ldots, n-c-2\}$ (this set is $\{1\}$ if $n-c-2 = 1$ and is empty if $n - c - 2 = 0$), 
exactly one of $\{n - c - 1, n - c\}$, and exactly one of $\{n - c + 1, \ldots, n\}$. Then
we have removed exactly $2c$ facets, and ${\rm link}_\Delta(\{1, \ldots,n - c - 2\})$ is $1-$dimensional and disconnected, meaning $\Delta$ is not Cohen--Macaulay, by Reisner's Criterion. Since every vertex decomposable complex is Cohen-Macaulay, $\Delta$ is not vertex decomposable.

\end{example}

\begin{remark}\label{dim1case}
In the situation where $\dim\Delta = 1$, so that $c = n-2$, we may view $\Delta$ as a graph on vertex set $[n]$ and our main result has a graph theoretic interpretation.  Since $\Delta$ is assumed to be pure, we consider graphs with no isolated vertex. In this context, Example \ref{example tight} yields a complex $\Delta$ that is the union of a complete graph with an isolated edge, and the link being considered is ${\rm link}_\Delta(\emptyset) = \Delta$. Indeed,
from Lemma\ \ref{vdprop} we know that a pure $1$-dimensional complex is vertex decomposable if and only if it is connected. Hence in this case 
Theorem \ref{maintheorem} is equivalent to the following graph theoretic statement:  an edge-cut of the complete graph $K_n$ into nontrivial parts has size at least $2c = 2n - 4$.  We refer to \cite{Diestel:2017} for definitions from graph theory.
\end{remark}

\begin{remark} \label{remark matroid}
One can also ask whether the conclusion of Theorem \ref{maintheorem} can be strengthened, that is whether there is some property {\it stronger} than vertex decomposable that is implied by $\Delta$ having ``enough facets''. Two well-studied classes of vertex decomposable simplicial complexes are independence complexes of matroids and shifted complexes.

Recall that a pure simplicial complex $\Delta$ is the independence complex of a {\it matroid} if it satisfies the following exchange property: if $F$ and $G$ are facets of $\Delta$ such that $v \in F \backslash G$, then there exists a $w \in G \backslash F$ such that $(F \backslash \{v\}) \cup \{w\}$ is again a facet of $\Delta$. A pure simplicial complex $\Delta$ is {\it shifted} if there exists a (re)labeling of the vertex set $V = [n] = \{1, \dots, n\}$ such that whenever $\{v_1, v_2, \dots, v_k\}$ is a face of $\Delta$, replacing any $v_i$ by a vertex with a smaller label again results in a face of $\Delta$.

We show that neither of these properties are implied by the assumptions in Theorem \ref{maintheorem}.  As an example, consider the complete graph $K_4$ as a $1$-dimensional simplex on $n = 4$ vertices (so that $c=2$).  Theorem \ref{maintheorem} implies that if we remove 3 or fewer facets (edges) and are left with no isolated vertices, we end up with a vertex decomposable complex. It turns out that we can remove just two facets from $K_4$ to destroy these stronger properties.

On the one hand, we can delete the facets $\{1,2\}$ and $\{1,3\}$, resulting in a complex that is easily seen not to be a matroid complex.
On the other hand, if we start with $K_4$ and delete the facets $\{1,2\}$ and $\{3,4\}$ we are left with a 4-cycle, which is not shifted. Indeed the class of shifted $1$-dimensional complexes coincide with the class of {\em threshold} graphs \cite{Klivans:2007}, which are known to be chordal \cite{Mahadev/Peled:1995}.
\end{remark}

We can also recover 
a result similar to
\cite[Theorem 2.1]{Terai/Yoshida:2006}. 
In particular, if we assume
$\Delta$ has ``even more'' facets,
then $\Delta$ is automatically pure, so
it must be vertex decomposable.

\begin{lemma}\label{morefacets}
    Let $\Delta$ be a simplicial complex on vertex set $[n]$
of codimension $c$.  Suppose $\Delta$
has at least $\binom{n}{c}-c$ facets 
of dimension $\dim\Delta = n-c-1$.  Then $\Delta$
is pure.
\end{lemma}

\begin{proof}
If $\Delta$ has a facet $F$ of dimension smaller than $\dim \Delta$, then $F$ can be extended to a set $S$ containing $\dim \Delta$ vertices. Namely, if $G$ is a facet of dimension $\dim\Delta$, then $F \not \subset G$, giving $|G \backslash F| \geq \dim \Delta - \dim F + 1$. We obtain the set $S$ by adding $\dim\Delta - \dim F$ vertices from $G \backslash F$ to $F$. Now, $S$ is contained in $c+1$ sets of size $n-c$, none of which are facets of $\Delta$ (as $S$ properly contains the facet $F$). This is a contradiction to the assumption that there are at least $\binom{n}{c}-c$ facets of dimension $\dim\Delta = n-c-1$.  
\end{proof}

We now derive a strengthening of 
\cite[Theorem 2.1]{Terai/Yoshida:2006}:

\begin{corollary} \label{corollary even more}
Let $\Delta$ be a simplicial complex on vertex set $[n]$
of codimension $c$.  If $\Delta$
has at least $\binom{n}{c}-c$ facets 
of dimension $\dim\Delta = n-c-1$, then $\Delta$
is vertex decomposable.
\end{corollary}

\begin{proof}
    If $c=0$, then $\Delta$ is a 
    simplex and the
    result follows immediately.
    So, suppose that $c \geq 1$.
    By Lemma \ref{morefacets},
    $\Delta$ is a pure simplicial
    complex of codimension $c$
    with at least $\binom{n}{c}-c \geq 
    \binom{n}{c}-2c+1$ facets.  
    Now apply
    Theorem \ref{maintheorem}.
\end{proof}

 Terai and Yoshida showed that Stanley-Reisner rings having large multiplicities are Cohen-Macaulay \cite[Theorem 3.1]{Terai/Yoshida:2006}, by proving that the dual ideal of
 the Stanley-Reisner ideal of $\Delta$
 has a linear resolution \cite[Theorem 3.3]{Terai/Yoshida:2006}. We can also recover this dual version as a direct consequence of our main result by employing Stanley-Reisner
theory and properties of Alexander
duality, 
and in particular Eagon and Reiner's classification of square-free monomial ideals with a
linear resolution. The following
corollary may be of
interest for those researchers
working with monomial ideals. We assume that 
the reader
is familiar with the Stanley-Reisner
correspondence; for 
undefined
terminology, see \cite{Miller/Sturmfels:2005}. 
The following result 
originally appeared as 
\cite[Theorem 3.3]{Terai/Yoshida:2006}.

\begin{corollary} \label{corollary linear}
Let $I$ be a square-free monomial ideal
of $R = \mathbb{K}[x_1,\ldots,x_n]$ with 
$\mathbb{K}$ an arbitary field, and suppose
that every generator has degree $c$.
If $I$ has at least $\binom{n}{c}-2c+1$ generators, then $I$ has a linear resolution.
\end{corollary}

\begin{proof}
    Given a square-free monomial
    $m = x_{i_1}\cdots x_{i_c}$ of
    degree $c$, set $F_m = [n] \setminus
    \{i_1,\ldots,i_c\}$.
    So, if $I = \langle m_1,\ldots,m_t \rangle
    $ is a square-free monomial ideal
     where each generator has degree $c$,
     then $\Delta = \langle F_{m_1},\ldots,
     F_{m_t} \rangle$ is a pure 
     simplicial complex of codimension $c$
     with at least $\binom{n}{c}-2c+1$ facets.
     By Theorem \ref{maintheorem},
     $\Delta$ is a vertex decomposable 
     simplicial complex, and consequently,
     $\Delta$ is Cohen-Macaulay.

     The ideal $I$ is the Alexander
     dual of $I_\Delta$, the Stanley-Reisner
     ideal of $\Delta$.  Since $\Delta$
     is Cohen-Macaulay, the Eagon-Reiner
     Theorem \cite{Eagon/Reiner:1998} implies
     that $I$ has a linear resolution.
\end{proof}

Our results also have applications to the study of {\it Simon's Conjecture} \cite{Simon}, a statement involving the structure of shellable complexes. To explain this connection we first review some relevant concepts.

\begin{definition}
 A pure $d$-dimensional simplicial complex $\Delta$ is {\it shellable} if there exists an ordering of its facets $F_1, F_2, \dots, F_s$ such that for all $k = 2,3, \dots, s$ the simplicial complex
\[\left ( \bigcup_{i=1}^{k-1} \langle F_i \rangle \right ) \cap \langle F_ k \rangle\]
\noindent
is pure of dimension $d-1$.  
\end{definition}
By convention the void complex $\emptyset$ and the empty complex $\{\emptyset\}$ are both shellable.  Note that any pure $d$-dimensional complex $\Delta$ on vertex set $[n]$ is a subcomplex of $\Delta_n^{(d)}$, the $d$-dimensional skeleton of the simplex on $[n]$. It is known that $\Delta_n^{(d)}$ is shellable, and one can ask if it is possible to produce a shelling of $\Delta_n^{(d)}$ that begins with a shelling of $\Delta$.

\begin{definition}
A pure $d$-dimensional simplicial complex $\Delta$ on vertex set $[n]$ is {\it shelling completable} if there exists a shelling of $\Delta$ that is the initial sequence of a shelling of $\Delta_n^{(d)}$.
\end{definition}

With these notions, we can formulate Simon's Conjecture as follows.

\begin{conjecture}\label{conjecture Simon}
Every shellable complex is shelling completable.
\end{conjecture}

In \cite{Coleman-Dochtermann-Geist-Oh} it is shown that every pure vertex decomposable complex is shelling completable.  Hence our Theorem \ref{maintheorem} implies that any counterexample to Simon's Conjecture (a pure shellable complex that is not shelling completable) must have ``not too many facets''. In particular if $\Delta$ has codimension $c$, then any counterexample
can have at most $\binom{n}{c}-2c$ facets.

There are other ways to detect the vertex decomposability of a simplicial complex $\Delta$ in terms of combinatorial data.
 In \cite{Lason:2013} Laso\'{n} provides a sufficient condition for vertex decomposability in terms of the $f$-vector of $\Delta$ (more specifically the number of facets and {\it ridges}). To describe these results, suppose $\Delta$ is a pure simplicial complex on vertex set $[n]$, with dimension 
$\dim\Delta= d $ and codimension $\cod \Delta = c =  n-d-1$. Recall that the {\it $f$-vector}
of $\Delta$ is $f(\Delta) = (f_0,f_1,\ldots,f_d)$ where $f_i$ is the number of faces of
$\Delta$ of dimension $i$.  Theorem \ref{maintheorem} implies that if 
$f_{d} \geq \binom{n}{c}-2c+1$, then
$\Delta$ is vertex decomposable.   

We say that a simplicial complex $\Delta$ is {\it extremal}
if $\Delta$ is pure of dimension $t$ with $r$
facets, and it has the minimum
number of faces of dimension $(t-1)$ 
among all simplicial complexes with $r$ faces of dimension $t$.  
As shown by \cite[Corollary 1]{Lason:2013}, an extremal simplicial complex can be
determined from its $f$-vector.  The main result from \cite{Lason:2013} is the following.

\begin{theorem}\cite[Theorem 5]{Lason:2013}
An extremal simplicial complex is vertex decomposable.
\end{theorem}

\begin{example} \label{example fvector}
To see that our result does not follow from Laso\'{n}'s work, let $n =6$ and consider
the pure two-dimensional simplicial complex
with facets 
$$\{1,2,3\},~\{1,4,5\},~\{1,2,6\},~\{1,4,6\},
~\{2,4,5\},~ \{3,4,5\},~\{3,5,6\}$$
and any other 9 facets of dimension 2, for
a total of 16 facets of dimension 2.
In this case, the codimension is $c = 3$, and
$16 \geq \binom{6}{3}-2\cdot 3 +1 = 15$, so, this
simplicial complex is vertex decomposable by
our Theorem \ref{maintheorem}.  By our
construction, this simplicial complex
has $f$-vector $(6,15,16)$ (our initial
choice of facets forces this simplicial
complex to contain all the $1$-dimensional
faces). However, this $f$-vector does not satisfy
\cite[Corollary 1]{Lason:2013}.
Indeed, there exists a pure simplicial complex
with $f$-vector $(6,14,16)$, namely
the simplicial complex which contains
all the facets of dimension two {\it except}
$\{1,5,6\},~ \{2,5,6\}, ~\{3,5,6\},$ and
$\{4,5,6\}$ (this simplicial complex
does not contain the face $\{5,6\}$). We see that our simplicial complex is not
extremal, and hence one can not deduce that it is vertex decomposable from \cite{Lason:2013}.
\end{example}

As a final consequence, we mention a connection
to geometrically vertex decomposable
ideals, as first defined by Klein and
Rajchgot \cite{KR2021}.   While we do not reproduce
the formal definition here, a geometrically
vertex decomposable ideal is a generalization
of the properties of the square-free monomial
ideals that corresponds to a vertex decomposable 
simplicial complex via the Stanley-Reisner
correspondence.  As shown in \cite[Proposition 2.14]{KR2021}, under a suitable 
lexicographical monomial order $<$, 
in some cases we can determine if an
ideal $I$ is a geometrically vertex decomposable ideal from its initial ideal 
${\rm in}_<(I)$.  In particular,
if ${\rm in}_<(I)$ is the Stanley-Reisner
ideal of a vertex decomposable simplicial complex,
then $I$ is geometrically vertex decomposable.  
By combining Klein and Rajchgot's result with our
Theorem \ref{maintheorem}, we get a new technique to
check if an ideal is geometrically vertex decomposable. Precisely, if there is
an appropriate monomial order $<$ such that
the initial ideal ${\rm in}_<(I)$ 
is the Stanley-Reisner ideal 
of a simplicial complex with ``enough'' facets, as determined by Theorem \ref{maintheorem}, then
$I$ is geometrically vertex decomposable.

%%%%%%%%%%%%%%%%%%%%%%%%%%%%%%%%%%%%%%%%%%%%%%%%%%%%%%%%%

\subsection*{Acknowledgements} 
The authors thank the anonymous referees for their helpful comments
and improvements.
They are grateful to Hailong Dao for his feedback and suggestions, and also thank Martin Milani\v{c} for answering questions regarding the graph theory literature.
This project began at the workshop ``Interactions 
Between Topological  Combinatorics and 
Combinatorial 
Commutative Algebra'' held at the
Banff International Research Station (BIRS) in April 2023.  
We thank BIRS and the organizers for the  opportunity to work together.
Dochtermann is partly supported by Simons Foundation Grant $\#964659$. 
Nair is partly supported by Dao's Simons Foundation Grant FND0077558.
Van Tuyl's research is supported by NSERC Discovery Grant 2019-05412. 
Woodroofe's research is supported in part by the Slovenian Research Agency (research program P1-0285 and research projects N1-0160, J1-2451, J1-3003 and J1-50000).

%%%%%%%%%%%%%%%%%%%%%%%%%%%%%%%%%%%%%%%%%%%%%%%

\bibliographystyle{amsplain}
\bibliography{arxiv2.bib}
%\newpage
\end{document}